\documentclass[12pt]{amsart}
\usepackage{amsmath,amsthm,amsfonts,amssymb,eucal}
\usepackage[notcite,notref]{showkeys}
\renewcommand {\a}{ \alpha }

\newcommand{\y}{\eta}

\newcommand{\m}{\mu}
\newcommand{\n}{\nu}

\newcommand{\g}{\gamma}
\newcommand{\G}{\Gamma}

\newcommand{\varf}{\varphi}
\renewcommand{\d}{\delta}

\newcommand{\D}{\Delta}

\renewcommand{\l}{\lambda}

\newcommand{\z}{\zeta}

\newcommand{\vart}{\vartheta}

\newcommand{\Om}{\Omega}
\newcommand{\R}{ \mathbb R}

\newcommand{\N}{ \mathbb N}

\newcommand{\Z}{ \mathbb Z}
\newcommand{\Sq}{ \mathbb S}
\newcommand{\CL}{\mathcal L}
\newcommand{\CB}{\mathcal B}
\newcommand{\CC}{\mathcal C}
\newcommand{\CD}{\mathcal D}
\newcommand{\CE}{\mathcal E}

\newcommand{\CH}{\mathcal H}
\newcommand{\CI}{\mathcal I}

\newcommand{\CX}{\mathcal X}

\newcommand{\CN}{\mathcal N}

\newcommand {\GA}{\mathfrak A}
\newcommand {\GB}{\mathfrak B}

\newcommand {\GH}{\mathfrak H}
\newcommand {\GJ}{\mathfrak J }
\newcommand {\GK}{\mathfrak K}

\newcommand {\ba}{\mathbf a}

\newcommand {\be}{\mathbf e}

\newcommand {\bm}{\mathbf m}

\newcommand {\bx}{\mathbf x}
\newcommand {\by}{\mathbf y}
\newcommand {\bb}{\mathbf b}

\newcommand {\BH}{\mathbf H}

\newcommand {\BL}{\mathbf L}
\newcommand {\BM}{\mathbf M}

\newcommand{\wt}{\widetilde}
\newcommand{\wh}{\widehat}

\DeclareMathOperator{\comp}{comp}

\newtheorem{thm}{Theorem}[section]

\newtheorem{prop}[thm]{Proposition}

\theoremstyle{definition}

\theoremstyle{remark}
\newtheorem{rem}[thm]{Remark}

\numberwithin{equation}{section}

\newcommand{\sh}{Schr\"odinger }
\newcommand{\vs}{\vskip0.2cm}
\newcommand{\pl}{\rm{pl}}

\newcommand{\ch}{\rm{ch}}

\newcommand{\fin}{\rm{fin}}

\newcommand{\w}{\infty}
\newcommand{\const}{\rm{const}}

\newcommand{\rad}{\rm{rad}}
\newcommand{\nrad}{\rm{nrad}}

\newcommand{\HH}[1]{\overset{\boldsymbol{\scriptscriptstyle{#1}}}{\BH}}
\newcommand{\VV}[1]{\overset{\boldsymbol{\scriptscriptstyle{#1}}}{V}}
\newcommand{\DD}[1]{\overset{\boldsymbol{\scriptscriptstyle{#1}}}{\D}}

\newcommand{\Sob}[1]{\overset{\boldsymbol{\scriptscriptstyle{#1}}} {\CH}}
\begin{document}
\title[Spectral estimates in dimension 2]
{On spectral estimates\\
 for the Schr\"odinger operators in global dimension 2}
\author[G. Rozenblum]{G. Rozenblum}
\address{Department of Mathematics \\
                        Chalmers University of Technology
                        and  The University of Gothenburg \\
                         S-412 96, Gothenburg, Sweden}
\email{grigori@chalmers.se}
\author[M. Solomyak]{M. Solomyak}
\address{Department of Mathematics
\\ Weizmann Institute of Science\\ Rehovot,  Israel}
\email{solom@wisdom.weizmann.ac.il}

\subjclass{47A75; 47B37, 34L15, 34L20}
\keywords{Eigenvalue  estimates, Schr\"odinger operator, Metric graphs, Local dimension.}
\dedicatory{TO BORIS MIKHAILOVICH MAKAROV, ON HIS 80-TH BIRTHDAY}

\begin{abstract}The problem of finding eigenvalue estimates for the Schr\"odinger operator turns out to be most complicated for the dimension $2.$ Some important results for this case have been obtained recently.  We discuss these results and establish their counterparts for the operator on the combinatorial and metric graphs corresponding to the lattice $\Z^2$.
\end{abstract}
\maketitle
\section{Introduction}\label{intro}For a self-adjoint operator $\BH$ in the Hilbert space $\GH$, whose negative spectrum is discrete,
we denote by $N_-(\BH)$ the dimension of its spectral projection that corresponds to the negative semi-axis $(-\w,0)$. In other words, $N_-(\BH)$ is the total number of all negative eigenvalues of
$\BH$, counted according to their multiplicities. For the \sh
operator $\BH=\BH_V=-\D-\a V$ (where $\a$ is a large parameter, {\it the coupling constant} and $V\ge0$) and its analogs the problem of obtaining estimates for   $N_-(\BH)$ has been attracting the interest of researchers for several last decades. In the case of the standard \sh operator in $\R^d$ with  $d\ge3$ the CLR estimate
 \begin{equation}\label{clr}
    N_-(-\D-\a V)\le C\a^{d/2}\int_{\R^d}V^{d/2}dx,\ d>2.
\end{equation}
 is sharp in order in $\a$ and in the function class for the potentials. It was obtained 40 years ago, and numerous generalizations had been found since then.  One of possible directions  for such generalizations concerns Schr\"odinger-like operators on structures that look globally as $\R^d$, but locally have a different dimension $\d$. The leading example here is the lattice $\Z^d$, which locally has dimension $\d=0$ but globally looks like $\R^d$. An exact explanation of the terms 'locally' and 'globally' in this context, as well as the corresponding results, can be found in \cite{RS}--\cite{RS11}, and in \cite{mv08}.

The case $d=2$ proves to be the most complicated and it is not completely understood up to now, both for the classical Schr\"odinger operator and its generalizations.

Recently,  several important results for operators in $\R^2$ were obtained. In the present paper we discuss some of these results and their counterparts for  operators with local dimension $0$ and $1$. To stress  close relations between three cases under consideration, we will denote the operators, functions etc. by the same symbols in all cases, marking the local dimension by the overset numeral, like in $\DD2$, or $\VV1$, when this is not clear from the context.  We concentrate ourselves on estimates having semiclassical order with respect to $\a$. In case of the global dimension $d=2$ this means that we are interested in estimates of the type
\begin{equation}\label{estim}
    N_-(\BH_{V})\le 1+\Phi(V)
\end{equation}
where the functional $\Phi(V)$ is homogeneous of order $1$ with respect to $V$, so that \eqref{estim}
automatically implies
\begin{equation}\label{estima}
    N_-(\BH_{\a V})\le 1+ \Phi(V)\a,\qquad\forall\a>0.
\end{equation}
Note that the term $1$ in \eqref{estim} and in \eqref{estima} reflects the well-known fact that
for all the cases considered the operator $\BH_{\a V}$ has at least one negative eigenvalue for any $\a>0$.

\section{The setting}\label{Setting}
We are interested in  estimates for $N_-(\BH)$ (in particular, in the conditions guaranteeing $N_-(\BH)<\w$) for three cases.

{\bf 1.}
$\GH=L^2(\R^2)$ and $\BH=-\DD2-\VV2$, i.e., it is the standard \sh operator  with a real-valued, non-positive potential $-\VV2$. Here $\d=2$.\vs

{\bf 2.} $\GH=\ell^2(\Z^2)$ and $\BH=-\DD0-\VV0$ is the discrete \sh operator with a
real-valued, non-positive potential $-\VV{0}$ defined on the lattice $\Z^2$. Here $\d=0$.\vs

{\bf 3.} The third case is intermediate between the previous two ones. Here we are dealing with the metric graph $\G_{\ch}$ that can be visualized as the union of the straight lines on $\R^2$, dividing the plain into the union of unit squares. We will call it \emph{"the chessboard mesh"}.
We give all the necessary detail about the graph $\G_{\ch}$ and about the corresponding \sh operator $\BH=\HH{1}_{\VV{1}}=-\DD{1}-\VV{1}$ in Section \ref{ael1}. Here we only note that for this graph $\d=1$. The operator $\HH{1}_{\VV{1}}$ acts as $-u''-\VV1u$ on each edge, and the functions from its domain meet some matching conditions at each vertex. The potential $\VV1\ge 0$ is a function, defined on the union of edges.\vs

The spectral nature of the Laplacians $\DD{2}$,  $\DD{0}$ and $\DD{1}$ is quite different. In particular, the operator $-\DD{0}$ is bounded in $\ell^2(\Z^2)$,  its spectrum is purely a.c. and fills $[0,4]$. The classical Laplacian $-\DD{2}$ is unbounded in $L^2(\R^2)$, its spectrum is purely a.c. and fills $[0,\w)$. The operator $-\DD{1}$ is unbounded as well. Its spectrum also fills $[0,\w)$, however, along with the a.c. component, it contains embedded eigenvalues at the points $\pi^2l^2$, $l=1,2,\dots$, and these eigenvalues have infinite multiplicity.
\begin{rem}
Alternatively to $N_-(\BH_V)$, one can look for estimates for the number of nonpositive eigenvalues, $N_{\le0}(\BH_V)$, like in \cite{mv08}, \cite{mv12}. Since there are no zero energy eigenfunctions of the unperturbed Laplacian, the quadratic form $-(Vu,u)$ is  negative definite on the null space of the operator $\BH_V$, and therefore the quadratic form $(\BH_Vu,u)-(Vu,u)=(\BH_{2V}u,u)$ is negative definite on the spectral subspace of $\BH_V$ corresponding to the nonpositive part of the spectrum (cf. the observation in \cite{mv08}). Thus,
$$N_{\le0}(\BH_V)\le N_-(\BH_{2V}).$$
Therefore, any estimate for the number of negative eigenvalues  carries over automatically  to a similar estimate for the number of nonpositive eigenvalues (with just a constant changed), so the former one is only formally weaker than the latter, provided we do not care for sharp constants.  Keeping this in mind, we, following the tradition, discuss estimates for $N_-(\BH_V)$ only.
\end{rem}

\section{The $\R^2$-case: Shargorodsky estimate}\label{SectSharg}

In spite of important differences, the estimates for the quantities $N_-(\HH{\d}_V)$ for $\d=0,1,2$ have much in common. For $d\ge3$ this was discovered in \cite{RS}-- \cite{RS10}. The situation in $\R^d,\ d>2,$ is governed by the
 CLR-inequality \eqref{clr}.
 A similar inequality is valid for the discrete \sh operator on $\Z^d$. It can be easily
derived from \eqref{clr}, see \cite{RS09}, but it also can be obtained independently \cite{levs, RS}. A CLR inequality for the $d$-dimensional analog of the chessboard mesh (its definition for an arbitrary $d\ge3$ is obvious)  has a similar form, with $V$ replaced by a certain effective discrete potential, see \cite{RS10}.

On the contrary, the case $d=2$ is not completely understood up to now. There are several upper estimates for $N_-(\HH{2})$ known. The estimate formulated below, the sharpest one known up to now, has been  obtained recently by  Shargorodsky \cite{Sh}, and is a refinement of earlier estimates in \cite{S94,LN,ls-nrad}. \vs

The estimate concerns the $\R^2$-case, and the overscript will be suppressed till the end of this section. The formulation of the estimate is rather complicated, since it makes use of function spaces that appear in the spectral theory not frequently. Let us present the necessary auxiliary material.

Below, $(r,\vart)$ stand for the polar coordinates in $\R^2$, and $\Sq$ denotes the unit circle $r=1$. Given a function $V$, such that $V(r,\cdot)\in L^1(\Sq)$ for almost all $r>0$, we introduce its radial and non-radial parts
\begin{gather*}
    V_{\rad}(r)=\frac1{2\pi}\int_{\Sq} V(r,\vart)d\vart;\qquad
    V_{\nrad}(r,\vart)=V(r,\vart)-V_{\rad}(r).
\end{gather*}

The conditions will be imposed on $V_{\rad}$ and on $V_{\nrad}$ separately.
For handling the radial part, we need a certain  auxiliary operator on the real line,
\begin{equation}\label{dim1}
    (\BM_{G}\varf)(t)=-\varf''(t)-G(t)\varf(t),\qquad \varf(0)=0,
\end{equation}
with the "effective potential"
\begin{equation}\label{effpot}
    G(t)=G_V(t)=e^{2|t|}V_{\rad}(e^t).
\end{equation}
Due to the condition $\varf(0)=0$ in \eqref{dim1}, the operator
$\BM_G$ is the direct orthogonal sum of two operators, each acting on the half-line. The sharp spectral estimates for $\BM_G$ can be given in terms of the number sequence (see Eq. (1.13) in \cite{ls-nrad})
\begin{equation}\label{sum}
   \wh{ \boldsymbol\z}(G)=\{\wh{\z_j}(G)\}_{j\ge0}:\qquad
     \wh{\z_0}(G)=\int_{D_0}G(t)dt, \quad
   \wh{\z_j}(G)=\int_{|t|\in D_j}|t|G(t)dt\quad  (j\in\N)
   \end{equation}
where $D_0=(-1,1)$ and $D_j=(e^{j-1},e^j)$ for $j\in\N$. The estimate is
\begin{equation}\label{est1}
    N_-(\BM_G)\le 1+ C\sup_{s>0}\left(s\#\{j:\wh{\z_j}G)>s\}\right),
\end{equation}
see Theorem 5.4, {\bf 2} in \cite{S12}.

Note that the functional appearing in the right-hand side of \eqref{est1} is nothing but the
quasi-norm of the sequence $\{\wh{\z_j}(G)\}$ in the "weak $\ell_1$-space" $\ell_{1,\w}$.

The condition on $V_{\nrad}$ is given in terms of the space $L_1\left(\R_+,\, L_{\GB}(\Sq)\right)$, i.e., the space of functions on the half-line with values in the space $L_{\GB}(\Sq)$. The latter is  the Orlicz space of functions on the unit circle, defined
by the $\CN$-function
 \begin{equation}\label{orlB}
    \GB(t)=(1+t)\ln(1+t)-t.
 \end{equation}
 See \cite{KR}, or \cite{RR} for the basics in Orlicz spaces, and in particular, for the definition of the norm in them. Some additional details are presented in Section \ref{weiest} below. There we need also the $\CN$-function
 \begin{equation}\label{orlA}
    \GA(t)=e^t-t-1,
 \end{equation}
 complementary to $\GB(t)$. In what follows, for brevity, we write $\GB(\Sq)=L_{\GB}(\Sq)$.\vs

We will suppose that $V_{\nrad}\in\CX= L_1\left(\R_+,\, \GB(\Sq)\right)$. The latter space is defined by the norm
\[ \|f\|_\CX=\int_{\R_+}\|f(r,\cdot)\|_{\GB(\Sq)}rdr.\]

The resulting estimate, see \cite{Sh}, is given by
\begin{thm}\label{Sharg}
Let $V\ge0$ be a potential on $\R^2$ and let $G=G_V$ be the corresponding effective potential as in \eqref{effpot}. Suppose that the sequence $\{\wh{\z_j}(G)\}$ belongs to
the space $\ell_{1,\w}$ and that $V_{\nrad}\in L_1\left(\R_+,\, \GB(\Sq)\right)$.
Then
\begin{equation}\label{est2}
    N_-(\BH_V)\le 1+C\left(\sup_{s>0}\left(s\#\{j:\z_j(G_V)>s\}\right)+\|V_{\nrad}\|_{L_1\left(\R_+,\ L\log L(\Sq)\right)}\right).
\end{equation}
\end{thm}

Still, this theorem gives only a sufficient condition for the semi-classical behavior $N_-(\BH_{\a V})=O(\a)$ in the large coupling constant regime. A number of  estimates, non-linear in $\a$, are also known, see \cite{mv12}.

For the radial potentials $V(x)=F(|x|)$ the estimate \eqref{est2} simplifies since the second term
in brackets disappears. It was established in \cite{ls-rad} that for such potentials
 this estimate gives not only sufficient, but also necessary condition for the semi-classical behavior. For arbitrary (i.e., not necessarily radial) potentials the following important result of a "negative" nature
 was recently established in \cite{gn}, Section 2.7.
 \begin{thm}\label{noest}
For the operator $\BH_V$ on $\R^2$ no estimate of the type
\begin{equation}\label{weight-est}
    N_-(\BH_V)\le \const+\int_{\R^2}VWdx
\end{equation}
can hold, provided the weight function $W$ is bounded in a neighborhood of at least one point.
\end{thm}

\section{The $\Z^2$-case: reduction to the $\R^2$-case}\label{Sect2to0}
In this Section we present the machinery for transferring spectral estimates from the operator $\HH{2}$ to $\HH{0}$. This simple approach has been already used in \cite{RS09} for the study of the \sh operator on $\Z^d, \  d\ge3$, however certain modifications are needed in the two-dimensional case.
\subsection{Interpolation. Sobolev space $\Sob0$.}\label{Sob0}
Consider the semi-norm on the space of functions on $\Z^2$
\begin{equation}\label{Alt1}
    \|u\|^2_{\Sob0}=
    \sum_{\bx\in\Z^2}\left(|u(\bx+e_1)-u(\bx)|^2+|u(\bx+e_2)-u(\bx)|^2\right),
\end{equation}
where $e_1=(1,0), \ e_2=(0,1)$.
It is an analog of the Dirichlet integral on $\R^2$.

  With any function $u(\bx), \bx\in \Z^2$  we associate the function $U_0(x)=(\GJ_0 u)(x)$,  $x\in\R^2$,
in the following way: at first, in each square with the vertices
 $\bx,\bx+e_1,\bx+e_2,\bx+e_1+e_2;\ \bx\in\Z^2,$ we take the function $U_0$ that is bi-linear, that is, linear in $x_1$ and
 in $x_2$ (separately), and coincides with $u$ at the vertices. Such function is unique. The following simple fact has been used in \cite{RS09}:

\begin{prop}\label{prop.interp} For any finitely supported function $u$ on $\Z^2$ the function  $U_0=\GJ_0 u$ belongs to the usual Sobolev space $H^1(\R^2)$ and
\begin{equation}\label{interp}
    \int_{\R^2}|\nabla U_0|dx\le C\|u\|^2_{\Sob0},\qquad U_0=\GJ_0 u.
\end{equation}
\end{prop}
We borrow the proof from \cite{RS09}.
\begin{proof} Consider the space $\CL(\CC)$ of bi-linear functions on  the
unit cell $\CC=[0,1]^2$. Clearly, $\dim\CL(\CC)=4$. On
$\CL(\CC)$ we consider the quadratic forms
\begin{equation*}
\wt Q[U_0;\CC]=\sum_{\bx,\by\in\{0,1\}^2\atop x\sim
y}|U_0(x)-U_0(y)|^2;\qquad \wt D[U_0;\CC]=\int_\CC|\nabla U(\xi)|^2d\xi.
\end{equation*}
These two quadratic forms vanish on the same subspace in
$\CL(\CC)$, the one consisting of constant functions. Therefore, they are
equivalent, i.e., with some $c,c'>0$ we have
\begin{equation}\label{cell}
    c\wt Q[U_0;\CC]\le\wt D[U_0;\CC]\le c'\wt Q[U_0;\CC].
\end{equation}
The function $U_0$ is compactly supported, continuous on the whole of $\R^2$, and smooth
inside each cell. Hence, $U_0\in  H^1(\R^2)$.
By adding up  inequalities of the form \eqref{cell} for all  cells $\CC+\bx,\
\bx\in\Z^2$, we arrive at \eqref{interp}.\end{proof}

Now we fix a smooth cut-off function $\psi(x),\ x\in \R^2,$ which vanishes in the neighborhood of $x=(0,0)$ and equals 1 for $|x|>\frac12$. We define the interpolation operator $\GJ$ by setting $U=\psi U_0=\psi\GJ_0 u$

The function $U$ differs from the corresponding function $U_0$ on four squares only, therefore it follows from \eqref{interp} that for all finitely supported functions $u$ on $\Z^2$, such that $u(0,0)=0$, the inequality similar to \eqref{interp} holds, with $U$ replacing $U_0$:

\begin{equation}\label{interp1}
    \int_{\R^2}|\nabla U|^2dx\le C\|u\|^2_{\Sob0},\qquad U=\GJ u.
\end{equation}

The next step is to obtain a weighted estimate for functions on the lattice.
\begin{prop}\label{DiscrHardy}For all functions $u(x), \ x\in \Z^2$ with finite support, satisfying $u(0,0)=0,$ the following discrete Hardy inequality holds:
\begin{equation}\label{HardyFin}
    \|u\|^2_{\Sob0}\ge C \sum_{\bx\in\Z^2, \bx\ne(0,0)}|u(\bx)|^2|\bx|^{-2}(\log (|\bx|+2))^{-2}.
\end{equation}
\end{prop}
\begin{proof}It follows from \eqref{interp} and the logarithmic Hardy inequality for functions in $\R^2$, vanishing near the origin, that for $U=\GJ u$, we have
\begin{equation}\label{HardyFin1}
     \|u\|^2_{\Sob0}\ge C\int |\nabla U|^2dx\ge C'\int |U(x)|^2|x|^{-2}(\log^2 (|x|+2))^{-1}dx.
\end{equation}
Taking into account that $U(0,0)=u(0,0)=0$ and that $U$ is piecewise bi-linear (with exception of 4 central squares), we can estimate from below the last integral on in \eqref{HardyFin1} by the sum of the values of the integrand in the lattice points, which produces the Hardy type inequality in question.
\end{proof}

\subsection{Main result for the $\Z^2$-case.}\label{carrover}
Given a discrete potential $V=\VV0\ge0$, we associate with it the piecewise-constant potential $\VV2=\CI(\VV0)$, assigning at each point $(x_1,x_2)\in\R^2$ the value
$\VV2(x_1,x_2)=\VV0([x_1], [x_2])$  ($[.]$, as usual, denotes the entire part of the number in brackets).

 \begin{prop}\label{PropCarry}Let $\VV0 \ge 0$ be a discrete potential on $\Z^2$ and $\VV2=\CI(\VV0)$. Then
\begin{equation}\label{Alt4}
    N_-(\DD0-\VV0)\le N_-(\DD2-\g\VV2),
\end{equation}
with some constant $\g>0$, not depending on $\VV0.$
\end{prop}
\begin{proof} Let first $N_-(\DD0-\VV0)=\bm<\infty$. This means that there exists a subspace $\CL\subset \ell^2(\Z^2)$, having dimension $\bm$ such that $\|u\|^2_{\!\!\Sob0}< \bb[u]$ for all $u\in \CL$, $u\ne 0$. By compactness, one may suppose that all functions in $\CL$ have support in a common compact set $\GK\subset \Z^2$.

Consider the set  $\boldsymbol{\CL}=\GJ \CL$ consisting of interpolants $U=\GJ u,\ u\in \CL$, where $\GJ$ is the interpolation operator, described in Subsection \ref{Sob0}. This is a space of functions on $\R^2$ contained in $\Sob2(\R^2)$, having the same dimension $\bm$.
One readily sees that $\|\nabla U\|_{L^2}^2 \le C_1 \|u\|^2_{\Sob0}$ and $\bb_{\VV2}[U]\ge C_2 \bb_{\VV0}[u]$. Therefore, for $U\in \boldsymbol{\CL},\  U\ne0$ we have
\begin{equation}\label{Alt5}
    \|\nabla U\|_{L^2}^2 -\g\a \int V|U|^2 dx< C_1(\|u\|^2_{\Sob0})- \bb[u])<0, \qquad \g= {C_1}C_2^{-1}.
\end{equation}
The last inequality means that $N_-(\DD2-\g \VV2)\ge \bm$, which proves \eqref{Alt4} for $\bm<\infty$.  In the case $\bm=\infty$, we can repeat the above reasoning for any finite-dimensional subspace $\CL$ and obtain that $N_-(\DD2-\g\VV2)$ exceeds any given natural number, in other words, that  $N_-(\DD2-\g\VV2)=\infty$.
\end{proof}
Note that we did not use Proposition \ref{DiscrHardy} in the proof.

\section{The  weighted estimate for $\Z^2$}\label{weiest}
As it follows from Proposition \ref{PropCarry}, any eigenvalue estimate for the usual \sh operator has its counterpart for the discrete one. The weak side of this approach is that
it gives the estimate for $\Z^2$-case in terms of the associated potential $\VV2=\CI\VV0$, rather than in terms of the original discrete potential $\VV0$,  and some features, in particular related to the circle symmetry, are hopelessly lost. In this connection, a recent result by Molchanov and Vainberg (see Theorem 6.1 in \cite{mv12}) deserves a special attention, since it is free from this defect.
\begin{thm}\label{MVain}
For any discrete potential $\VV0\ge0$ on $\Z^2$, the estimate holds:
\begin{equation}\label{MVest}
    N_-(\HH0_{\VV0}) \le 1+C\sum_{\bx\in\Z^2}\VV0(\bx)\log(2+|\bx|).
\end{equation}
\end{thm}
Due to Theorem \ref{noest} there can exist no corresponding $\R^2$-estimate, having a similar form. The aim of this section is to show that, nevertheless, \eqref{MVest} can be derived by the interpolation procedure from an existing $\R^2$-estimate, namely, from the result of \cite{S94}.

To formulate this latter result, we use the pair \eqref{orlB} and \eqref{orlA} of mutually complementary $\CN$-functions. For any measurable set $E\subset\R^2$ of finite Lebesgue measure the Orlicz space $L_{\GB}(E)$ is defined as the space of measurable functions $v$ on $E$, such that $\int_E\GB(|v(x)|) dx<\infty$. Since the $\CN$-function $\GB(t)$ meets the so-called $\D_2$-condition, see \cite{KR,RR},
such functions form a Banach space, and one of the equivalent norms in it is the \emph{averaged norm}  (introduced in \cite{S94}):
\begin{equation}\label{OrlAv}
    \|v\|_{\GB, E}^{(av)}=\sup\left\{\left|\int_E vg dx\right|:\int_{E}\GA(|g(x)|)dx\le |E| \right\}.
\end{equation}

Now, to formulate the estimate from \cite{S94}, we consider the partition of $\R^2$:
\begin{equation}\label{partitions}
    \Om_0=\{x:|x|\le1\}, \ \Om_k=\{x: 2^{k-1}\le|x|\le 2^k\}.
\end{equation}
With a given potential $\VV2\ge0$, we associate the  number sequence $\boldsymbol{\mu}(\VV2)=\{\m_k(\VV2)\}$, where \begin{equation}\label{m(VV2)}
 \m_k(\VV2)=\|\VV2\|_{\CB,\Om_k}^{(av)}.
 \end{equation}

 In these notations, by Theorem 3 in \cite{S94}, (more precisely, by its simplified version, see (32) there)
 \begin{equation}\label{MZ94}
    N_-(\DD2- \a \VV2)\le 1 +C \|\boldsymbol{\m}(\VV2)\|_{\ell^1}+\int \VV2(x)|\log |x|| dx.
 \end{equation}

 Now, using Proposition \ref{prop.interp}, we derive the estimate \eqref{MVest} from \eqref{MZ94}. To this end, for a given $\VV0$, we consider the piecewise constant potential $\VV2 =\CI(\VV0)$. The estimate  \eqref{MVest} follows from  \eqref{MZ94}  immediately as soon as we prove that

 \begin{equation}\label{z2}
    \|\boldsymbol{\m}(\VV2)\|_{\ell^1}\le C \sum_{\bx\in\Z^2}\VV0(\bx)\log(2+|\bx|).
 \end{equation}
 For $k$ fixed, we consider the set $E_k$ formed by all unit squares in the lattice, intersecting $\Om_k.$ Such sets form a covering of $\R^2$ with multiplicity 2, moreover, $|E_k|\asymp|\Om_k|$.
 So, since $\log(|\bx|+2)\asymp k+1$ for $\bx\in E_k$, to prove \eqref{z2}, it suffices to establish the inequality
 \begin{equation}\label{z3}
    \|\VV2\|_{\GB,E_k}^{(av)}\le C(k+1) \sum_{\bx\in E_k}\VV0(\bx).
 \end{equation} Further on, in this proof, all summations are performed over $\bx\in E_k\cap \Z^2$. For any unit square $Q$ (with vertices in $\Z^2$) and a nonnegative function $g$ we denote by $g_Q(x)$ the constant function on $Q$ which equals the mean value of $g$ over $Q$, so that $\int_Q g\VV2 dx=\int_Q g_Q\VV2dx$. By Jensen's inequality,
 \begin{equation*}
    \int_{Q}\GA(g(x))dx\ge\GA\left(\int_Q g(x)dx\right)=\int_{Q}\GA(g_Q(x))dx,
 \end{equation*}
 therefore, to obtain an upper estimate for the averaged  norm of $\VV2$, it is sufficient to maximize over piecewise constant functions $g$:
 \begin{equation}\label{z4}
    \|\VV2\|_{\GB,E_k}^{(av)}\le \sup\{\sum\VV0(\bx) f(\bx): \sum\GA(g(\bx))=|E_k|\}.
 \end{equation}
 By linearity, it is sufficient to estimate  the quantity in \eqref{z4} under the normalization $\sum \VV0(\bx)=1$, so we aim for estimating the quantity
 \begin{equation}\label{z5}
    \sup\{\sum \VV0(\bx) g(\bx): \sum\GA(g(\bx))=|E_k|\asymp C2^k\}
 \end{equation}
We can, instead, consider the expression $S(\VV0,g)=\sum \VV0(\bx) g(\bx)$
and maximize it over all collections $\{\VV0(\bx)\ge0\}, \{g(\bx)\ge0\}, $ under the conditions $\sum \VV0(\bx)=1, \sum \GA(g(\bx))=E_k.$ By compactness and continuity, a point of maximum must exist.
We state that  at the point of maximum, only one of $\VV0(\bx)$ is not zero. In fact, if, say, $\VV0(\bx_1),\VV0(\bx_2)\ne 0$
and $g(\bx_1)$ (or $g(\bx_2)$) is $0$, then we can change $\VV0(\bx_1)$ to 0 and add this $\VV0(\bx_1)$ to such $\VV0(\bx)$, for which $g(\bx)>0$, thus keeping $\sum \VV0(\bx)=1$ and increasing $S(\VV0,g)$. If both $g(\bx_1),g(\bx_2) $ are positive and, say,
$g(\bx_1)\ge g(\bx_2)$, we change $\VV0(\bx_1)\mapsto \VV0(\bx_1)+\VV0(\bx_2)$, $\VV0(\bx_2)\mapsto 0$ and thus increase $S(\VV0,g)$, keeping $\sum \VV0(\bx)=1$.

So, the function $\VV2$, maximizing the norm $\|\VV2\|_{\CB, E_k}^{av}$ among those piecewise constant nonnegative functions on $E_k$, for which $\int_{E_k}\VV2 dx=\sum_{\bx\in E_k\cap \Z^2} \VV0(\bx)=1$, is the function  supported just on one square and having value 1 there. For such function, the averaged Orlicz norm over the domain $E_k$ equals
\begin{equation}\label{z6}
    \|\VV2\|_{\GB,E_k}^{av}=\VV2 {\GA}^{-1}(|E_k|)\asymp \log |E_k|\asymp k+1.
\end{equation}
Summing up such inequalities over all $k$,  we arrive at \eqref{z3} and this finishes the (alternative) proof of \eqref{MVest}.

\begin{rem}\label{omaloe}
Along with the estimate \eqref{MVest}, one has
\begin{equation}\label{omal}
    N_-(\DD0-\a\VV0)=o(\a), \qquad\a\to\w.
\end{equation}
\end{rem}
Indeed, this is certainly true for the finitely supported potentials. Since such potentials
form a dense set in the cone of non-negative elements in the weighted space $\ell^1(\Z^2,\log(2+|x|)$, the result extends, by continuity, to all such weights.

The details can be easily restored by the analogy with \cite{RS10}, see the  proof of Theorem 2.3 there.

\section{The chessboard mesh}\label{ael}
\subsection{Basic definitions}\label{ael1}
First of all, we present the detailed description of the metric graph $\G_{\ch}$ that we call the chessboard mesh, see Section \ref{Setting}. Its set of vertices is $\Z^2$ (considered as naturally imbedded into $\R^2$) and the set $\CE$ of edges consists of the intervals $\be_\n$ of length one, connecting the neighboring vertices. Here $\n$ in the index stands for the pair $\n=(\bx,\by)=(\by,\bx)$ of the vertices which are the endpoints of the edge. The measure $dz$ on $\G_{\ch}$ is induced by the Lebesgue measure on the edges and our main Hilbert space is $L^2(\G_{\ch}), dz)=\oplus_{\be\in\CE} L^2(\be,dz)$.  A function $u$ on $\G_{\ch}$ belongs to the Sobolev space $H^1(\G_{\ch})$ if it is continuous on $\G_{\ch}$, lies in $H^1(\be)$ for each $\be\in\CE$, and
\begin{equation}\label{sobE}
    \|u\|_{H^1(\G_{\ch})}^2=\int_{\G_{\ch}}(|u'(z)|^2+|u(z)|^2)dz=
    \sum_{\be\in\CE}\int_{\be}(|u'(z)|^2+|u(z)|^2)dz<\infty.
\end{equation}

The (minus) Laplacian on $\G_{\ch}$ is defined as the self-adjoint
operator in $L^2(\G_{\ch})$, associated with the quadratic form $\ba_{\G_{\ch}}[u]:=\int_{\G_{\ch}}|u'(z)|^2dz$ considered on the form-domain $H^1(\G_{\ch})$. Its operator domain, $\CD(\D)$, can be explicitly described as follows: $u\in\CD(\D)$ if and only if $u$ is continuous on $\G_{\ch}$, $u''(z)\in L^2(\be)$ on each edge $\be\in\CE$, the Kirchhoff matching condition is fulfilled at each vertex $\bx\in\Z^2$ (that is, the sum of outgoing derivatives of $u$ at the point $\bx$ equals zero), and, finally,
\[\sum_{\be\in\CE}\int_{\be} (|u''|^2+|u|^2)dz<\infty.\]
On this domain the Laplacian acts as $\D u=u''$ on each edge; see \cite{Kuch} for more detail on the Laplacian on metric trees.

The spectrum of $-\D$ coincides with $[0,\infty)$. It consists of the a.c. component, filling the same interval, and the embedded eigenvalues $\l_n=\pi^2n^2$, each of infinite multiplicity (see \cite{Ku2}).

The \sh operator $-\D-V$ on $\G_{\ch}$ is standardly defined as the form-sum, the potential $V\ge0$ being a measurable function on $\G_{\ch}$. The assumptions we impose on $V$
guarantee that the quadratic form $\ba_{\G_{\ch}}[u]-\int_{\G_{\ch}} V|u|^2dz$ is bounded from below and closed on $H^1(\G_{\ch})$.

For the spectral analysis of the \sh operator on $\G_{\ch}$ it is convenient to consider two pre-Hilbert spaces that are linear subspaces in
the space $H^1_{\comp}(\G_{\ch})$ of all compactly supported functions in $H^1(\G_{\ch})$. One of them,
$H^1_{\comp,\pl}(\G_{\ch})$, is formed by  functions linear on each edge $\be\in\CE$;
the subscript $\pl$ stands for `piecewise-linear'. Any function
$\varf\in H^1_{\comp,\pl}(\G_{\ch})$ is determined by its values $\varf(\bx)$ at the
vertices. Given a  lattice function $u=\{u(\bx)\},\ \bx\in\Z^2$ with finite support , we denote by
$Ju $ the piecewise-linear extension of $u$ to $\G_{\ch}$ i e., the (unique) function in $H^1_{\comp,\pl}(\G_{\ch})$, such that
$(Ju)(\bx)=u(\bx)$ for all $ \bx\in\Z^2$. The mapping
$J$ defines an isometry between the pre-Hilbert spaces $H^1_{\comp,\pl}(\G_{\ch})$
equipped with the metric $\ba_{\G_{\ch}}$, and $\CH_{\fin}(\Z^2)$ of finitely supported lattice functions, equipped with the
metric \eqref{Alt1}. By means of  this isometry
we identify   these pre-Hilbert spaces.

Another subspace is $H^1_{\comp,\CD}$  ($\CD$ hints for \emph{Dirichlet}) consists of all functions
$\varf\in H^1_{\comp}(\G_{\ch})$, such that $\varf(\bx)=0$ for all $\bx\in \Z^2$.
It is clear that
\begin{equation}\label{prehilbert}
   H^1_{\comp}=H^1_{\comp,\pl}\oplus H^1_{\comp,\CD}
\end{equation}
(the orthogonal decomposition in the metric  $\ba_{\G_{\ch}}$). We will
denote by $\varf_{\pl}$ and $\varf_\CD$ the components of a given element
$\varf\in H^1_{\comp}(\G_{\ch})$ with respect to this decomposition.

\begin{prop} \label{Hardy1Prop}For any function $\varf\in H^1_{\comp} $, such that $\varf(0,0)=0$,  the following  Hardy type inequality holds:
\begin{equation}\label{Hardy1}
    \ba_{\G_{\ch}}[\varf]\ge C\int_{\G_{\ch}}|\varf(z)|^2(|z|^2(\log^2(|z|+2)))^{-1}dz.
\end{equation}
\end{prop}
\begin{proof}It is sufficient to prove the inequality for $\varf_{\pl}$ and $\varf_\CD$ separately. For $\varf_{\pl}$, by the isometry $J,$ it obviously follows from the inequality \eqref{HardyFin}, and for $\varf_\CD$  - from the usual Hardy inequality applied to each of intervals $\be\in\CE.$
\end{proof}

Now we can consider the space $\Sob1$ of functions on $\G_{\ch}$, which is the closure of $H^1_{\comp}$ in the metric $\ba_\G$. By Proposition \ref{Hardy1Prop}, this is a space of functions, embedded in $L^2(\G_{\ch})$ with weight $|z|^2(\log^2(|z|+2))^{-1}$. The closures of the terms in \eqref{prehilbert} are Hilbert spaces $\Sob1_{\!\pl}$ and $\Sob1_{\!\CD}$ respectively, giving an orthogonal decomposition of $\Sob1,$
\begin{equation}\label{hilbert}
  \Sob1=\Sob1_{\!\pl}\oplus \Sob1_{\!\CD}.
\end{equation}

\subsection{Spectral estimates on $\G_{\ch}$.}\label{SectMetricEst}
The passage from eigenvalue estimates for $\HH0$ to the ones for $\HH1$ was elaborated in detail in \cite{RS10} for general graphs. Here we explain, how it works in our particular case.  Let a potential $\VV1$ on $\G_{\ch}$ be given. For each edge $\be$ in $\G_{\ch}$, we consider the Sturm-Liouville operator in $L^2(\be)$, $\HH1=\BH_{\be,\CD}(\VV1)=-\frac{d^2}{dz^2}-\VV1(z)$, with Dirichlet  boundary condition at the endpoints of $\be$. The direct sum of these operators is denoted $\HH\CD$; this operator corresponds to the second term in the decomposition \eqref{hilbert}. Another operator,  stemming from the first term in the  decomposition \eqref{hilbert}, is the operator $\HH0$ on the lattice $\Z^2$, with the effective potential
\begin{equation}\label{effective}
    \VV0(\bx)=\sum_{\be\ni\bx}\y(\be),
\end{equation}
where $\y(\be,\VV1)=\int_{\be}\VV1(z)dz$. We denote by $\pmb{\y}(\VV1)$ the sequence $\{\y(\be,\VV1)\}$.
By the general result of \cite{RS10},  Sect.4.3, we have
\begin{equation}\label{decoupling}
    N_-(\HH1_{\frac12\VV1})\le N_-(\HH\CD)+N_-(\HH0).
\end{equation}
\begin{thm}Denote by $\rho(\be)$ the distance from the origin to the nearest to the origin point of $\be$.  Suppose that the sum
\begin{equation}\label{BL}
\BL(\VV1)=\sum_{\be}\y(\be, \VV1)\log(2+\rho(\be))
\end{equation}
is finite. Then $N_-(\HH1)\le 1+C\BL(\VV1)$ with a constant $C$ not depending on $V$.
\end{thm}
\begin{proof} We have to estimate both terms in \eqref{decoupling}. For the second one the result of Theorem \ref{MVain} applies, since the sum $\sum\VV0(\bx)\log(2+|\bx|)$ is estimated from both sides by the sum in \eqref{BL}. For the first term in \eqref{decoupling}, we can apply Lemma 4.2, $2^{\circ}$ in \cite{RS10} for the particular value $q=1$, which gives
\begin{equation*}
    N_-(\HH0)\le C\|\pmb{\y}(\VV1)\|_{\ell^1_w}
\end{equation*}
The last expression is dominated by the right-hand side of \eqref{BL}, and we are done.
\end{proof}
The sum in \eqref{BL} can be upper estimated  by the integral $\int_\G \VV1(z)\log(2+|z|) dz$, and this leads us to the final result.
\begin{thm}\label{intdest}Suppose that the integral $\BM(\VV1)=\int_\G \VV1(z)\log(2+|z|) dz$ converges. Then
\begin{equation*}N_-(\HH1)\le 1+C\BM(\VV1).
\end{equation*}
\end{thm}
We note here that Theorem \ref{MVain} and Theorem \ref{intdest} demonstrate that in local dimension 0 and 1 there cannot exist a forbidding result similar to Theorem \ref{noest}.

\end{document}